\documentclass[a4paper,12pt]{amsart}

\usepackage{color}

\newtheorem{thm}{Theorem}

\begin{document}

\title{Compact homogeneous lcK manifolds are Vaisman}
\author{P. Gauduchon, A. Moroianu and L. Ornea}
\date{\today}
\maketitle
\begin{abstract} We prove that any compact homogeneous locally
  conformally K\"ahler manifold has parallel Lee form.
\end{abstract}
\bigskip

{\it Theorem \ref{thm-lck} 
 below has been claimed by K. Hasegawa and Y. Kamishima in
\cite{HK1} and \cite{HK2} and a partial result also appeared  in
\cite{MO}.  At the time of writing, it is not clear to us that the arguments presented in  \cite{HK1} and \cite{HK2} are complete. We here present a complete  simple proof of
this result. }
\bigskip

A Hermitian manifold $(M, g, J)$ is called {\it locally conformally
  K\"ahler} --- lcK for short --- if, in some neighbourhood of  any
  point of $M$,  the Hermitian  structure can be made K\"ahler by some
  conformal change of the metric. Equivalently, $(M, g, J)$ is lcK if
  there exists a {\it closed} real $1$-form $\theta$, called the {\it
    Lee form} of the Hermitian structure, such that the {\it K\"ahler form} 
$\omega := g (J
  \cdot, \cdot)$ satisfies:
\begin{equation} \label{lcK} d \omega = \theta \wedge \omega. 
\end{equation}
A lcK Hermitian structure is called {\it strictly}  lcK if the Lee form
$\theta$ is not identically zero,  and {\it Vaisman} if 
$\theta$ is a non-zero parallel $1$-form with respect to the Levi-Civita 
connection of the metric $g$. Equivalently, since $\theta$ is closed,  
a strictly lcK manifold is Vaisman if the {\it Lee vector field} $\xi = \theta
^{\sharp}$ is Killing, i.e. $\mathcal{L} _{\xi} g = 0$, where
$\mathcal{L} _{X}$ denote the Lie derivative along a vector field
$X$. In general,
the Lee vector field $\xi$ and the vector field 
$J \xi$ (sometimes called the {\it Reeb
  vector field}  of the lcK structure) are  neither Killing, nor holomorphic
(meaning that $\mathcal{L} _{\xi} J = 0$), but we always have
$\mathcal{L} _{J \xi} \omega = 0$, since 
$\mathcal{L} _{J \xi} \omega = - d \theta + \theta (J \xi) \,
\omega + \theta \wedge \theta = 0$. 
\smallskip

By a {\it compact homogeneous lcK manifold} we mean 
a compact, connected, strictly lcK manifold $(M, g, J, \omega)$,
equipped with a transitive  and effective left-action of 
a (compact, connected) Lie group
$G$, which preserves the whole Hermitian structure, i.e. the
Riemannian metric $g$, the (integrable) complex structure $J$
and the $2$-form $\omega$. We then have:
\begin{thm} \label{thm-lck} Any compact homogeneous lcK manifold $(M,
  g, J, \omega)$ is  Vaisman.
\end{thm}
Before starting the proof of Theorem \ref{thm-lck}, we 
recall a number of general,
well-known facts, concerning a
compact homogeneous manifold $(M, g)$, equipped with a left-action, 
effective and transitive, of a 
(connected) compact Lie group $G$. Without loss of generality, we
assume that the stabilizer $G _x$ of any point $x$ of $M$ in $G$ is
connected. 
We denote by $\mathfrak{g}$ the Lie algebra of $G$ and we fix an ${\rm
  Ad} _G$-invariant positive definite inner product, $B$, on
$\mathfrak{g}$. The induced (infinitesimal) action of $\mathfrak{g}$
is an injective linear map $\sigma: {\sf a}
\mapsto \hat{\sf a}$ from $\mathfrak{g}$ to the space, ${\rm Vect} (M)$, of
(smooth) vector fields on $M$, defined by: 
\begin{equation} \label{inf-action} \hat{\sf a}   (x) 
= \frac{d}{dt} \Big\arrowvert_{t = 0}
  {\rm exp} (t \, {\sf a}) \cdot x, \end{equation}
for any ${\sf a}$ in $\mathfrak{g}$ and any $x$ in $M$, where ${\rm
  exp}:\mathfrak{g} \to G$ denotes the exponential map and $\cdot$ the
action of $G$ on $M$. Since the $G$-action on $M$ is a left-action,
$\sigma$ is an {\it anti-isomorphism} from $\mathfrak{g}$ to
$\hat{\mathfrak{g}}$,  equipped with the usual bracket of vector
fields: $[\hat{\sf a}, \hat{\sf b}] = - \widehat{[{\sf a}, {\sf b}]}$. 
\smallskip

We denote by $\textsf{Inv}$ the space  of $G$-invariant  vector fields
on $M$, which is a Lie subalgebra of ${\rm Vect} (M)$. Any element $Z$
of  $\textsf{Inv}$ commutes with all elements of
$\hat{\mathfrak{g}}$. In particular 
\begin{equation} {\sf Inv} \cap \hat{\mathfrak{g}} =
  \hat{\mathfrak{c}}, \end{equation} 
where $\hat{\mathfrak{c}}$ denotes the image in $\hat{\mathfrak{g}}$ of
the center, $\mathfrak{c}$, of $\mathfrak{g}$. 
\smallskip

For any $x$ in $M$, we denote by $\sigma _x$ the map ${\sf a} \mapsto
\hat{\sf a} (x)$, from $\mathfrak{g}$ to the tangent space $T _x M$
of $M$ at $x$, and by $\sigma _x ^*$ its metric adjoint,  from $(T _x M,
g _x)$ to $(\mathfrak{g}, B)$, so that:
\begin{equation} B (\sigma _x ^* (X), {\sf a}) = g _x (X, \sigma _x
  (\sf{a})), \end{equation}
for any $X$ in $T _x M$ and any ${\sf a}$ in $\mathfrak{g}$. The
kernel, $\mathfrak{g} _x$, of $\sigma _x$, is a Lie subalgebra of
$\mathfrak{g}$, namely the Lie algebra of the stabilizer, $G _x$, of
$x$ in $G$, whereas the image of $\sigma _x ^*$ in $\mathfrak{g}$ 
is the $B$-orthogonal
complement, $\mathfrak{g} _x ^{\perp}$, of $\mathfrak{g} _x$. 
The Lie
algebra $\mathfrak{g} _x$ acts on $T _x M$ by ${\sf a} \cdot X = [\tilde{X},  \hat{\sf a}] (x) = \left(D ^g _X \hat{\sf a}\right) (x)$ 
for any ${\sf a}$ in $\mathfrak{g} _x$ and any $X$ in 
$T _x M$, where $\tilde{X}$ here stands for any local vector field
around $x$ such that $\tilde{X} (x) = X$ and $D ^g$ denotes the
Levi-Civita connection of $g$.  We then have
\begin{equation} \label{gx-equi} 
\sigma _x [{\sf a}, {\sf b}] = {\sf a} \cdot \sigma _x
  ({\sf b}), \qquad \sigma _x ^* ({\sf a} \cdot X) = [{\sf a}, \sigma _x
  ^* (X)], \end{equation}
for any $\sf{a}$ in $\mathfrak{g} _x$, any ${\sf b}$ in $\mathfrak{g}$
and any $X$ in $T _x M$. 
\smallskip

For any chosen point $x$
in $M$, the evaluation map ${\sf ev} _x: Z \mapsto {\sf ev} _x
(Z) := Z (x)$, from 
${\sf Inv}$ to $T _x M$, is injective, and 
$Z (x)$ belongs to the space, denoted by $(T _x M) ^{\mathfrak{g} _x}$, of those
$X$ in $T _x M$ such that 
${\sf a} \cdot X = 0$ for any ${\sf a}$ in $\mathfrak{g}
_x$; conversely, any $X$ in $(T _x M) ^{\mathfrak{g} _x}$ is equal to
$Z (x)$ for a unique $Z$ in ${\sf Inv}$. Then, ${\sf ev} _x$ is a linear {\it
  isomorphism} from ${\sf Inv}$ to $(T _x M) ^{\mathfrak{g} _x}$,
whose inverse is denoted by ${\sf ev} _x ^{-1}$. 
On the other
hand, for any 
$X$ in $(T _x M) ^{\mathfrak{g} _x}$, 
we have $X = \hat{\sf b} (x)$ for some ${\sf b}$ in
$\mathfrak{g}$,  which is uniquely defined up to the addition of 
an element of
$\mathfrak{g} _x$ and is such that $[{\sf a}, {\sf b}]$ belongs to 
$\mathfrak{g} _x$ for
any ${\sf a}$ in $\mathfrak{g} _x$, meaning that $\sf{b}$ belongs to
the normalizer, ${\rm N} _{\mathfrak{g}} (\mathfrak{g} _x)$,  of $\mathfrak{g}
  _x$ in $\mathfrak{g}$; conversely, for any $\sf{b}$ in 
$N _{\mathfrak{g}} (\mathfrak{g} _x)$, $\hat{\sf b} (x)$ is the value at
  $x$ of a (unique) element of ${\sf Inv}$. For any chosen $x$ in $M$,
  we thus get a linear isomorphism
\begin{equation} \label{Inv-id} {\sf Inv} = {\rm N} _{\mathfrak{g}} 
(\mathfrak{g}
  _x)/\mathfrak{g} _x, \end{equation}
according to which ${\sf b} \mod{\mathfrak{g} _x}$ is identified with
${\sf ev} _x ^{-1} (\sigma _x ({\sf b}))$, for any ${\sf b}$ in 
${\rm N} _{\mathfrak{g}} (\mathfrak{g} _x)$.   This isomorphism is 
actually a Lie algebra isomorphism. 

\smallskip

In general, the
rank, ${\sf rk} \, \mathfrak{l}$, 
of a Lie algebra $\mathfrak{l}$ of compact type 
(= the Lie algebra of a compact Lie
group) can be defined as the dimension of a maximal abelian Lie
subalgebra. In particular, ${\sf rk} \, {\rm N} _{\mathfrak{g}} (\mathfrak{g}
_x) \leq  {\sf rk} \, \mathfrak{g}$ and it then follows from 
(\ref{Inv-id}) that 
\begin{equation} \label{rank} 
{\sf rk} \, {\sf Inv} \leq  {\sf rk} \, \mathfrak{g} - {\sf
    rk} \, \mathfrak{g} _x. \end{equation}

\smallskip

\begin{proof}[Proof of Theorem \ref{thm-lck}]
We now assume that the homogeneous Riemannian 
manifold $(M, g)$ is equipped with a compatible $G$-invariant 
lcK Hermitian structure as
explained above.  In particular, 
${\sf Inv}$ is 
$J$-invariant, since $J$ is $G$-invariant, and contains the Lee vector
field $\xi$ and $J \xi$ (notice however that ${\sf Inv}$ is 
{\it not} a priori a complex Lie algebra, as $J$ is not preserved 
in general by
the elements of ${\sf Inv}$). 
Since the $G$-action on $M$ preserves the K\"ahler form $\omega$, it
preserves $d \omega$ as well, hence also the Lee form $\theta$. Since,
moreover, $d \theta = 0$, it follows that $\theta (\hat{\sf a})$ is
{\it constant} and that $\theta ([\hat{\sf a}, \hat{\sf b}]) = 0$, for any 
$\hat{\sf a}, \hat{\sf b}$ in $\hat{\mathfrak{g}}$. Alternatively,
$\theta$ determines an element, $\tilde{\theta}$, of $\mathfrak{g}
^*$, 
defined by 
$\tilde{\theta} ({\sf a}) := \theta (\hat{\sf a})$, 
which vanishes on the derived Lie subalgebra
$[\mathfrak{g}, \mathfrak{g}]$. Being of compact type, $\mathfrak{g}$
splits as $\mathfrak{g} = \mathfrak{c} \oplus
  \mathfrak{s}$, where $\mathfrak{s} = [\mathfrak{g}, \mathfrak{g}]$
  is the semi-simple part of $\mathfrak{g}$. We infer the existence of
  ${\sf t}$ in $\mathfrak{c}$ such that $\tilde{\theta} ({\sf t}) =
  1$, whose image in 
$\hat{\mathfrak{c}} = {\sf
   Inv} \cap \hat{\mathfrak{g}}$ will be denoted by $T$, instead of
 $\hat{\sf t}$; we
 then have $\theta (T) = 1$.  
 Since $\theta$ is closed, the
 kernel, $\mathfrak{g} _0$, of $\tilde{\theta}$ in $\mathfrak{g}$ 
is a Lie subalgebra of
 $\mathfrak{g}$ and we get the following $B$-orthogonal decomposition 
\begin{equation} \mathfrak{g} = \langle {\sf t}  \rangle \, \oplus
      \mathfrak{g} _0, \end{equation} 
where $\langle {\sf t}  \rangle$ denotes the $1$-dimensional subspace
of $\mathfrak{g}$ generated by ${\sf t}$. 
The $1$-form $\psi$ defined by
\begin{equation} \label{psi} \psi   := - \iota _{T} \omega,
\end{equation}
is $G$-invariant, in particular $\mathcal{L} _T \psi = 0$, from which we
 infer, by using (\ref{lcK}):  $d \psi = \iota _T d \omega = 
 \theta (T) \, \omega + \theta \wedge \psi$, hence
\begin{equation} \label{omega-psi} \omega = d \psi - \theta \wedge
  \psi. \end{equation} 
Notice that (\ref{omega-psi}) implies that $\theta \wedge \psi \neq 0$ 
at each point of $M$, since $\omega$ is everywhere non-degenerate,
whereas $\iota _T d \psi = 0$; it follows that the $G$-invariant vector
fields $T, J \xi$ are independent at each point of $M$. Since $[T, J
\xi] = \mathcal{L} _T (J \xi) = 0$, $T$ and $J \xi$ generate a
$2$-dimensional abelian Lie subalgebra of ${\sf Inv}$, so that
\begin{equation} \label{rank-geq} 
{\sf rk} \, {\sf Inv} \geq 2. \end{equation}
\smallskip

From (\ref{omega-psi}), we readily infer 
that $\mathfrak{c}$ is either $1$-dimensional, generated
by $T$, or $2$-dimensional, generated by $T$ and by $J \xi$, which is
then an element of $\hat{\mathfrak{c}} = {\sf Inv} \cap
\hat{\mathfrak{g}}$. Indeed, suppose that $\mathfrak{c}$ is of
dimension greater than $1$. There then exists ${\sf v} \neq 0$ in 
$\mathfrak{c}$
such that $\tilde{\theta} ({\sf v}) = 0$. 
Again, we denote by $V$ its image in
$\hat{\mathfrak{c}} = {\sf Inv} \cap \hat{\mathfrak{g}}$, instead of
$\hat{\sf v}$. 
Since $\psi$ is $G$-invariant, we have
$\mathcal{L} _V \psi = 0$, so that, by (\ref{omega-psi}), $\iota _V
\omega = \psi (V) \, \theta$. This implies that $V$ is then equal to a
(non-zero) multiple of $J \xi$. 

\smallskip

Choose  any $x$ in $M$  and denote by ${\sf w}_x$ the element of $\mathfrak{g}$
defined by
\begin{equation}\label{W} {\sf w} _x =  \sigma _x ^* \big((J T) (x)\big), 
\end{equation}
so that $B ({\sf w}_x, {\sf a})   =  g _x (JT, \sigma _x ({\sf a})) = 
- \psi  (\sigma
_x ({\sf a}))= - \psi  (\hat{\sf a} (x))$, for any ${\sf a}$ in
${\mathfrak{g}}$, where $\psi$ is the $G$-invariant $1$-form defined
by (\ref{psi}). 
 In particular, ${\sf w}_x$ sits in $\mathfrak{g} _0$,
since $\psi (T) = 0$, and is $B$-orthogonal to $\mathfrak{g} _x$. 
Denote by $C _{\mathfrak{g} _0} ({\sf w}_x)$ the
commutator of ${\sf w}_x$ in $\mathfrak{g} _0$ and 
pick any ${\sf a}$ in $\mathfrak{g} _0$. Then, ${\sf a}$ belongs to 
$C _{\mathfrak{g} _0} ({\sf w}_x)$ if and only 
if $B ([{\sf w}_x, \sf{a}], \sf{b}) =
0$ for any $\sf{b}$ in $\mathfrak{g} _0$.  Since  
$B ([{\sf w}_x, {\sf a}], {\sf b}) = B ({\sf w}_x, [{\sf a}, {\sf b}]) 
=  \psi  ([\hat{\sf a}, \hat{\sf b}](x) )= 
 d \psi (\hat{\sf a} (x), \hat{\sf b}(x))$, this occurs
if and only if $\iota _{\hat{\sf a} (x)} d \psi = 0$, if and only if 
$\iota _{\hat{\sf a} (x)} \omega - \psi  (\hat{\sf a} (x)) \, \theta
(x) = 0$
(by using (\ref{omega-psi}) and $\theta (\hat{\sf a}) = 0$). We
eventually get that ${\sf a}$ belongs to $C _{\mathfrak{g} _0} ({\sf w}_x)$ if
and only if $\hat{\sf a} (x) = - \psi  (\hat{\sf a}(x)) \, (J \xi) (x)
=  B ({\sf w}_x, {\sf a}) \, (J \xi) (x)$. We
infer that $\sigma _x ({\sf w}_x) =  B ({\sf w}_x, {\sf w}_x) \, J
\xi (x)$ 
and that
\begin{equation} \label{Cw} C _{\mathfrak{g} _0} ({\sf w}_x) 
= \langle {\sf w}_x \rangle \,
  \oplus \mathfrak{g} _x, \end{equation}
where the sum is $B$-orthogonal.  Since ${\sf w}_x$ is contained in a
maximal abelian Lie subalgebra of $\mathfrak{g} _0$, which is
contained in $C _{\mathfrak{g} _0} ({\sf w}_x)$, we have: 
${\sf rk}\, C _{\mathfrak{g} _0} ({\sf w}_x) = {\sf rk} \, \mathfrak{g} _0
= {\sf rk} \, \mathfrak{g} - 1$. It then follows from (\ref{Cw})
that 
\begin{equation} \label{rkggx} 
{\sf rk} \, \mathfrak{g} _x = {\sf rk} \,
  \mathfrak{g} - 2. \end{equation}
By using  (\ref{rank}), we infer ${\sf rk} \, {\sf Inv} \leq 2$, hence,
by (\ref{rank-geq}):
\begin{equation} {\sf rk} \, {\sf Inv} = 2. \end{equation}

\smallskip

Being a Lie algebra of compact type of rank $2$, with non-trivial
center (since it contains $T$),  
${\sf Inv}$ is 
isomorphic either to the abelian Lie algebra $\mathbb{R} \oplus
\mathbb{R}$ 
or to the unitary Lie algebra $\mathfrak{u} _2 = 
\mathbb{R} \oplus \mathfrak{su} _2$. Notice that
the latter case can only occur if $\mathfrak{c} = \langle {\sf t} \rangle$,
since $\hat{\mathfrak{c}} \cong \mathfrak{c}$ is contained in ${\sf
  Inv}$, whereas the center of $\mathfrak{u} _2$ is of dimension $1$. 

\smallskip

\noindent {\bf Case 1}. 
If ${\sf Inv} \cong \mathbb{R} \oplus \mathbb{R}$, it is 
generated by $T$ and $J \xi$, which, as already observed, are
independent at each point of $M$.   
Since $J T$ belongs
to ${\sf Inv}$, we thus have 
\begin{equation} \label{JT} J T  = a \, T + b \, J \xi, \end{equation}
for some real numbers $a, b$, with $b \neq 0$. Now, $T$ preserves
$\omega$ and $J$, as do any vector field in $\hat{\mathfrak{g}}$, 
so $J T$ preserves $J$, since $J$ is integrable. It
follows that $J \xi = \frac{1}{b} (J T - a \, T)$ preserves $J$ as
well. We already observed that $J \xi$
preserves $\omega$ for {\it any} lcK structure: $J \xi$  is then a Killing
vector field with respect to $g$. It then follows from (\ref{JT}) that
$J T$ is a Killing vector field. Finally, $\xi = \frac{1}{b} (T + a \,
JT)$ is also a Killing vector field, meaning that 
the lcK structure is Vaisman. 

\smallskip

\noindent {\bf Case 2}. It remains to deal with the case when 
$\mathfrak{c} = \langle {\sf t}  \rangle$ is 
$1$-dimensional and ${\sf Inv} \cong 
\mathfrak{u} _2 = \mathbb{R}  \, \oplus \, \mathfrak{su}
_2$. For convenience, we normalize $B$ so that $B ({\sf t}, {\sf t}) =
1$ and $B
({\sf w}_x, {\sf w}_x) = 1$, where ${\sf w}_x$ has been defined by
(\ref{W}),  
so that $\sigma _x ({\sf w}_x) = 
 (J \xi) (x)$ (where $x$ is any chosen point in $M$). 
 Denote by $e _0$ a generator of
the center $\mathbb{R}$ of $\mathfrak{u} _2$, and by $e _1, e _2, e _3$
a triple of generators of $\mathfrak{su}_ 2$, such that $[e _2, e _3] = e
_1$, $[e _2, e _3] = e _1$, $[e _3, e _1] = e _2$. Via (\ref{Inv-id}),
we may identify $e _0$
with ${\sf t} \mod{\mathfrak{g} _x}$ and $e _1$, say, with ${\sf w}_x
  \mod{\mathfrak{g} _x}$, in order that the restriction of $\omega$ --- given 
by (\ref{omega-psi}) --- to ${\rm N}  _{\mathfrak{g}} (\mathfrak{g}
_x)/\mathfrak{g} _x$  coincide  with the standard form $\omega _0 = e _0
\wedge e _1 +  e _2 \wedge e _3$;  we here identify $e _0, e _1, e _2, e
_3$ with their $B$-duals in $\mathfrak{g} ^*$, so that $\theta = e
_0$, 
$\psi (x) = - e _1$ and $(d \psi) (x) = e _2 \wedge e _3$ (by
identifying $\theta$ with $\tilde{\theta}$ and, similarly, $\psi (x)$
with ${\sf a} \mapsto \psi (\hat{\sf a} (x))$ in
$\left(\mathfrak{g}/\mathfrak{g} _x\right) ^*$).  
It remains to determine the complex
structure $J$ of ${\sf Inv}$ in terms of the generators $e _0, e
_1, e _2, e _3$. Without loss of generality, the generators $X, Y$ of 
the corresponding $2$-dimensional complex space 
$\Theta ^{(0, 1)} _J$ of elements of type $(0, 1)$ in  
${\sf Inv}  \otimes \mathbb{C} 
= \mathbb{C}\,  e _0 \oplus  \mathbb{C} \, e _1 
\oplus \mathbb{C} \, e _2 \oplus \mathbb{C} \, e _3$ can be chosen of
the form
$X = e _0 + \sum _{i = 1} ^3 a _i \, e _i, \quad
 Y = \sum _{i = 1} ^3 b _i \, e _i$, 
where $a_i, b _i$ are complex numbers, which must
satisfy the following three conditions: 
\begin{enumerate}

\item[(i)]  $\Theta _J ^{(0, 1)}$ is $\omega$-{\it Lagrangian}, i.e.  
$\omega (X, Y) = 0$, 
where $\omega$ is extended to a $\mathbb{C}$-bilinear form on 
$(\mathbb{R} \oplus \mathfrak{su} _2) \otimes \mathbb{C}$; 
\smallskip

\item[(ii)] $\Theta _J ^{(0, 1)}$ is {\it involutive}, meaning that 
$[X, Y] = \lambda \, X + \mu \, Y$, 
for some complex numbers $\lambda, \mu$; 
\smallskip

\item[(iii)]  $g := \omega _0 (\cdot, J
\cdot)$ must be positive definite on $\mathbb{R} \oplus \mathfrak{su}
_2$, meaning that 
$i \, \omega (Z, \bar{Z}) > 0$, 
for any $Z$ in $\Theta ^{(0, 1)} _J$. 
\end{enumerate}
It is easily checked that the
first condition (i) is expressed by 
\begin{equation} \label{C1} b _1  +  a_2 \, b _3  - a _3  \, b _2 = 0, \end{equation}
whereas the integrability condition (ii) implies $\lambda = 0$ and is
then equivalent to the system
\begin{equation} \label{C2} \begin{split} 
& - \mu \, b _1  - a _3 \, b _2 + a _2  \, b _3 = 0, \\ & 
 a _3  \, b _1 - \mu \, b _2 - a _1  \, b _3 = 0, \\ &
a _1  \, b_2 - a _2 \, b _1 - \mu \, b _3 = 0. \end{split} \end{equation} 
If $b_1 \neq 0$, we infer from (\ref{C1}) that $\mu = -1$, so that,  the
system (\ref{C2}) has a non-trivial solution in $b_1, b_2, b_3$ if and only
if $a_1, a_2, a_3$ are related by $\sum _{i = 1} ^3 a _i ^2 + 1 = 0$,
the solution then being 
$b _1 = a_1 \, a_3 - a_2$, $b_2 = a_2 \, a_3 + a_1$,  $b_3  
= a _3 ^2 + 1$.
We thus get $Y = a_3 \, \sum _{i = 1} ^3 a _i \, e _i - a_2 \, e _1
+ a_1\, e_2 + e _3 = - a _3 \, e _0 - a_2 \, e _1 + a_1 \, e _2 +
e _3 \mod{X}$. It follows that $\Theta ^{(0, 1)} _J$ meets the first two
conditions (i) and (ii), with $b _1 \neq 0$, if and only it is generated
by $X, Y$ of the form
$X = e _0 + \sum _{i = 1} ^3 a _i \, e _i, \quad
 Y = - a _3 \, e _0  - a _2 \, e _1 + a _1 \, e _2 + e _3$, 
with $\sum _{i = 1} ^3 a _i ^2 + 1 = 0$. As for the positivity 
condition (iii), we easily compute
$\omega _0 (X, \bar{X}) = \bar{a}_1
  - a_1 + a _2 \, \bar{a}_3 - \bar{a}_2 \, a_3 =  - \omega _0 (Y, \bar{Y})$, 
which shows that  (iii) is actually {\it never} 
satisfied if $b_1 \neq 0$. 
We thus have $b_1 = 0$, which, by (\ref{C1}),  implies
$a_2 b _3 - a _3 b _2 = 0$, 
so that $Y = b _2 \, e _2 + b_3 \, e _3$, whereas $X = e _0 + a_1 \, e
_1 \mod{Y}$; by
changing the notation,  
$\Theta ^{(0, 1)} _J$ is then generated by
$X := e _0 + a _1 \, e _1$ and $Y := b _2 \, e _2 + b _3 \, e
  _3$. 
Moreover, since $[X, Y] = a _1 \, (- b_3 \, e _2 + b _2 \, e _3)$, 
the integrability condition (ii) is satisfied if  and
only if $b _2 = k \, b_3$ and $b_3 = - k \, b_2$, for some complex number $k$,
which must be equal to $\pm i$. If $k = i$, we have 
$Y = e _2 - i \, e _3$,  hence $i \, \omega _0 (Y, \bar{Y}) = - 2$,
which is negative. We 
thus have $k = - i$, hence $Y = e _2 +  i \, e _3$,  whereas 
$i \, \omega _0 (X, \bar{X}) = 2 \, \mathfrak{Im}{(a)}$, by setting $a
_1 = a$, 
so that  $\mathfrak{Im}{(a)} > 0$. 
The only suitable complex structures
$J$ on 
$\mathbb{R} \oplus \mathfrak{su} _2$  are therefore of the form 
\begin{equation} \begin{split} & 
J e _0 = \frac{\mathfrak{Re}{(a)}}{\mathfrak{Im}{(a)}} \,
  e _0 + \frac{|a| ^2}{\mathfrak{Im}{(a)}} \, e _1, \quad J e
    _1 = - \frac{1}{\mathfrak{Im}{(a)}} \, e _0 -
    \frac{\mathfrak{Re}{(a)}}{\mathfrak{Im}{(a)}} \, e _1, \\
& J e _2 = e _3, \qquad J e _3 = - e _2, \end{split} \end{equation}
with $\mathfrak{Im}{(a)} > 0$. 
Since $e _0 = {\sf t} \mod{\mathfrak{g} _x}$ and 
$e _1 =  {\sf w}_x \mod{\mathfrak{g} _x}$ represent  $T$ and $J \xi$
respectively in ${\sf
  Inv}$, via (\ref{Inv-id}), it follows that
(\ref{JT}) 
is again  satisfied in
${\sf Inv}$, with $b = \frac{|a|^2}{\mathfrak{Im}{(a)}} > 0$; 
we then conclude as before that $\xi$ is Killing, i.e. that 
the lcK structure is Vaisman. 
\end{proof}


\begin{thebibliography}{10}

\bibitem{HK1} K. Hasegawa and Y. Kamishima, {\it 
Locally conformally K\"ahler structures on homogeneous spaces},
arXiv:1101.3693v10, 12 Jun 2013

\bibitem{HK2} K. Hasegawa and Y. Kamishima, {\it 
Compact Homogeneous Locally Conformally K\"ahler Manifolds}, 
arXiv:1312.2202v1, 8 Dec 2013

\bibitem{MO} A. Moroianu and L. Ornea, {\it Homogeneous 
locally conformally K\"ahler manifolds}, arXiv:1311.0671v1, 4 Nov 2013
\end{thebibliography}
\end{document}